\documentclass[12pt]{article}
\usepackage{amssymb,amsfonts,amsthm}
\usepackage{amsmath}
\usepackage{amstext}
\usepackage{color}
\usepackage[latin1]{inputenc}
\usepackage[active]{srcltx}
\def\openC{{\rm C\kern-.18cm\vrule width.8pt height 7pt depth-.2pt \kern.18cm}}
\def\openN{{{\rm I}\kern-.16em {\rm N}}}
\def\openR{{{\rm I}\kern-.16em {\rm R}}}
\def\openT{{{\rm T}\kern-.42em {\rm T}}}
\def\openZ{{{\rm Z}\kern-.28em{\rm Z}}}

\newtheorem{thm}{Theorem}[section]

\newtheorem{lem}[thm]{Lemma}
\newtheorem{prop}[thm]{Proposition}

\theoremstyle{definition}

\begin{document}

\title{
{\textbf{{Differentiable positive definite functions on two-point homogeneous spaces}}} \vspace{-4pt}
\author{\sc V. S. Barbosa and V. A. Menegatto}
}
\date{}
\maketitle \vspace{-30pt}
\bigskip

\begin{center}
\parbox{13 cm}{\small In this paper we study continuous kernels on compact two point homogeneous spaces which are positive definite and zonal (isotropic).\ Such kernels were characterized by R. Gangolli some forty years ago and are very useful for solving scattered data interpolation problems on the spaces.\ In the case the space is the $d$-dimensional unit sphere, J. Ziegel showed in 2013 that the radial part of a continuous positive definite and zonal kernel is continuously differentiable up to order $\lfloor (d-1)/2 \rfloor$ in the interior of its domain.\ The main issue here is to obtain a similar result for all the other compact two point homogeneous spaces.}
\end{center}

\noindent{\bf Keywords:} Positive definite kernels; Isotropic kernels; Homogeneous spaces; Jacobi polynomials; Differentiability

\thispagestyle{empty}

%
%

\section{Introduction}\label{s1}

Let $\mathbb{M}^d$ denote a $d$ dimensional compact two-point homogeneous space.\ It is well known that spaces of this type belong to one of the following categories (\cite{wang}):
the unit spheres $S^d$, $d=1,2,\ldots$, the real projective spaces $\mathbb{P}^d(\mathbb{R})$, $d=2,3,\ldots$, the complex projective spaces
$\mathbb{P}^d(\mathbb{C})$, $d=4,6,\ldots$, the quaternionic projective spaces $\mathbb{P}^d(\mathbb{H})$, $d=8,12,\ldots$, and the Cayley projective plane $\mathbb{P}^{d}(Cay)$, $d=16$.\ Standard references containing all the basics about two point homogeneous spaces that will be needed here are \cite{kush1,platonov} and others mentioned there.

In this paper, we will deal with real, continuous, positive definite and zonal (isotropic) kernels on $\mathbb{M}^d$.\ The positive definiteness of a kernel $K$ on $\mathbb{M}^d$ will be the standard one: it requires that
$$
\sum_{\mu,\nu=1}^n c_\mu c_\nu K(x_\mu,x_\nu) \geq 0,
$$
whenever $n$ is a positive integer, $x_1,x_2, \ldots, x_n$ are distinct points on $\mathbb{M}^d$ and $c_1,c_2, \ldots, c_n$ are real scalars.\ The continuity of $K$ can be defined through
the usual (geodesic) distance on $\mathbb{M}^d$, here denoted by $|xy|$, $x, y \in \mathbb{M}^d$.\ We will assume such distance is normalized so that all geodesics on $\mathbb{M}^d$ have the same length $2\pi$.\ Since $\mathbb{M}^d$ possesses a group of motions $G_d$ which takes any
pair of points $(x,y)$ to $(z,w)$ when $|xy|$=$|zw|$, zonality of a kernel $K$ on $\mathbb{M}^d$ will refer to the property
$$
K(x,y)=K(Ax,Ay), \quad x,y \in \mathbb{M}^d, \quad A \in G_d.
$$

A zonal kernel $K$ on $\mathbb{M}^d$ can be written in the form
$$
K(x,y)=K_r^d(\cos |xy|/2), \quad x,y \in \mathbb{M}^d,
$$
for some function $K_r^d: [-1,1] \to \mathbb{R}$, the {\em radial} or {\em isotropic part} of $K$.\ A result due to
Gangolli (\cite{gangolli}) established that a continuous zonal kernel $K$ on $\mathbb{M}^d$ is positive definite if and only if
\begin{equation}\label{PDM}
K_r^d(t)=\sum_{k=0}^{\infty}a_k^{(d-2)/2,\beta} P_k^{(d-2)/2,\beta}(t), \quad t \in [-1,1],
\end{equation}
in which $a_k^{(d-2)/2,\beta} \in [0,\infty)$, $k\in \mathbb{Z}_+$ and $\sum_{k=0}^{\infty}a_k^{(d-2)/2,\beta} P_k^{(d-2)/2,\beta}(1) <\infty$.\ Here, $\beta=(d-2)/2, -1/2, 0, 1, 3$, depending on the respective category $\mathbb{M}^d$ belongs to, among the five we have mentioned in the beginning of the paper.\ The symbol $P_k^{(d-2)/2,\beta}$ stands for the Jacobi polynomial of degree $k$ associated with the pair $((d-2)/2,\beta)$.

Gneiting (\cite{gneiting}) conjectured that the radial part of a continuous, positive definite and zonal kernel on $S^d$ is continuously differentiable in $(-1,1)$ up to order
$\lfloor(d-1)/2\rfloor$ (largest integer not greater than $(d-1)/2$).\ The conjecture was ratified by Ziegel (\cite{ziegel}) who also proved that the differentiability
order in Gneiting's conjecture is best possible in the case $d$ odd.\ In other words, she proved that if $d$ is odd, there exists
a continuous, positive definite and zonal kernel $K$ on $S^d$ for which the $\lfloor (d-1)/2\rfloor$ derivative of $K_r^d$ is not continuously differentiable.\ In
addition, she analyzed some specific examples to show that the one side derivatives of $K_r^d$ at the extreme points -1 or 1 can either take finite values or be infinite.

Menegatto (\cite{meneg}) added to Ziegel's results establishing similar results in the complex setting, that is, replacing
the unit sphere $S^d$ with the unit sphere in $\mathbb{C}^d$ and allowing the positive definite functions to assume complex values.\ The radial part of a positive definite kernel on complex spheres depend upon a complex variable $z$ and its conjugate $\overline{z}$.\ As so, in the complex setting, derivatives can be considered with respect to these two variables.\ The deduction of the results in this complex version demanded quite a number of changes in the procedure used in $S^d$, some of them not obvious.\ While Ziegel's arguments were based upon recurrence formulas for Fourier-Gegenbauer coefficients of certain continuous functions on $[-1,1]$, Menegatto's invoked some similar properties for the double indexed coefficients of a continuous function on the unit disk with respect to disk (Zernike) polynomials.

This is the point where we state the main result to be proved in this paper, a first step extension of the results described above to compact two point homogeneous spaces.

\begin{thm} \label{main} Let $d$ be a sufficiently large positive integer.\ If $K$ is a continuous, positive definite and zonal kernel on $\mathbb{M}^d$, then the radial part $K_r^d$ of $K$ is continuously differentiable on $(-1,1)$.\ The derivative $(K_r^d)'$
of $K_r^d$ in $(-1,1)$ satisfies a relation of the form
$$
(1-t^2)(K_r^d)'(t) = f_1(t) - f_2(t),\quad t\in(-1,1),
$$
in which $f_1$ and $f_2$ are the radial parts of two continuous, positive definite and zonal kernels on some compact two point homogeneous space $\mathbb{M}$ which is isometrically embedded in $\mathbb{M}^d$.\ The specifics on $d$ and $\mathbb{M}$ in each case are these ones:\\
$(i)$ $\mathbb{M}^d=S^d$: $d\geq 3$ and $\mathbb{M}=S^{d-2}$;\\
$(ii)$ $\mathbb{M}^d=\mathbb{P}^d(\mathbb{R})$: $d\geq 3$ and $\mathbb{M}=\mathbb{P}^{d-2}(\mathbb{R})$;\\
$(iii)$ $\mathbb{M}^d=\mathbb{P}^d(\mathbb{C})$: $d\geq 4$ and $\mathbb{M}=\mathbb{P}^{d-2}(\mathbb{C})$;\\
$(iv)$ $\mathbb{M}^d=\mathbb{P}^d(\mathbb{H})$: $d\geq 8$, $\mathbb{M}=\mathbb{P}^{d/2-2}(\mathbb{C})$, when $d \in 8\mathbb{Z}_++8$ and $\mathbb{M}=\mathbb{P}^{d/2}(\mathbb{C})$, when $d\in 8\mathbb{Z}_+ +12$;\\
$(v)$ $\mathbb{M}^d=\mathbb{P}^{16}(Cay)$: $\mathbb{M}=S^2$.
\end{thm}

The proof of this result depends upon tricky arguments involving a well-known three-term relation for Jacobi polynomials along with some other properties.\ The statement of the theorem also takes into account isometric isomorphisms among spaces not appearing in our initial list and certain spheres (see Section 5 for the due explanations).

More than half of the paper is concerned with the establishment of basic assumptions under which certain functions possessing Fourier-Jacobi expansions are differentiable.\ The results are proved in greater generality, that is, the expansions extrapolate the format in Gangolli's theorem: the upper indeex $(d-2)/2$ is replaced with a real number $\alpha>-1$ while $\beta$ is kept as general as possible.\ Basically, the assumptions involve the nonnegativity of some Fourier-Jacobi coefficients of the function and the convergence of certain series of Fourier-Jacobi coefficients also attached to the function.\ The discussion involving the extension of Theorem \ref{main} to higher order derivatives will be delayed to a separate section at the end of the paper.

We have arranged the paper in the following way.\ In Section 2, we recall some basic facts about Jacobi polynomials, some specific properties of inner product spaces where these polynomials form orthogonal systems and connections among Fourier-Jacobi coefficients coming from expansions of functions on such spaces.\ In Sections 3 and 4, we discuss upon the differentiability of Fourier-Jacobi expansions and deduce formulas to compute the derivatives using a three-term relation for derivatives of Jacobi polynomials.\ The contents in these two sections are similar, the major difference being the convergence assumptions on the series of Fourier-Jacobi coefficients adopted in each case.\ In Section 5, we prove several separated theorems that, together, ratify all the assertions in Theorem \ref{main}.


\section{Jacobi polynomials and Fourier-Jacobi expansions}

For  $\alpha,\beta>-1$, the set $\{ P_m^{\alpha,\beta}: m\in \mathbb{Z}_+\}$ of Jacobi polynomials associated to the pair $(\alpha,\beta)$ is orthogonal on $[-1,1]$ in the sense that
$$
\int_{-1}^1 P_m^{\alpha,\beta}(t)P_n^{\alpha,\beta}(t)(1-t)^{\alpha}(1+t)^{\beta}dt = \delta_{m,n} h_m^{\alpha,\beta}
$$
where
$$
h_m^{\alpha,\beta} = \frac{2^{\alpha+\beta+1}}{2m+\alpha+\beta+1}\frac{\Gamma(m+\alpha+1)\Gamma(m+\beta+1)}{\Gamma(m+1)\Gamma(m+\alpha+\beta+1)},\quad m\in \mathbb{Z}_+.
$$
Here, and in many other places in the paper, $\Gamma$ stands for the usual gamma function.\ Let us write $L_1^{\alpha,\beta}[-1,1]$ to denote the set of all measurable functions $f$ in $[-1,1]$ which are integrable with respect to the weight $(1-t)^\alpha (1+t)^\beta$, that is, for which
$$
\|f\|_1^{\alpha,\beta}:=\int_{-1}^1|f(t)|(1-t)^\alpha (1+t)^\beta dt < \infty.
$$
Every $f$ in $L_1^{\alpha,\beta}[-1,1]$ has a formal Fourier-Jacobi series
$$
f \sim \sum_{n=0}^\infty  a_n^{\alpha,\beta} R_n^{\alpha,\beta}
$$
where $R_n^{\alpha,\beta}:=P_n^{\alpha,\beta}/P_n^{\alpha,\beta}(1)$ and
$$
a_n^{\alpha,\beta}=a_n^{\alpha,\beta}(f):=\frac{[P_n^{\alpha,\beta}(1)]^2}{h_n^{\alpha,\beta}}\int_{-1}^1 f(t) R_n^{(\alpha,\beta)}(t)(1-t)^\alpha(1+t)^\beta dt.
$$
We observe that
$$
P_n^{\alpha,\beta}(1)=\left( \begin{array}{c} n+\alpha\\ n \end{array}\right):= \frac{\Gamma(\alpha+n+1)}{n!\Gamma(\alpha+1)}, \quad n \in \mathbb{Z}_+,\quad \alpha,\beta >-1.
$$
Returning to Gangolli's result, it is promptly seen that a continuous and zonal kernel $K$ on $\mathbb{M}^d$ is positive definite if and only if its radial part $K_r^d$ has a convergent Fourier-Jacobi series with respect to the normalized Jacobi polynomials $\{R_n^{(q-2)/2,\beta}: n\in \mathbb{Z}_+\}$, in which all
coefficients $a_{k}^{(d-2)/2,\beta}$ are nonnegative and $\sum_{n=0}^\infty  a_n^{(d-2)/2,\beta}<\infty$, with $\beta$ agreeing with the five categories of spaces we have mentioned before.\

Since the radial part of a continuous, positive definite and zonal kernel on $\mathbb{M}^d$ is given by convergent series of Jacobi polynomials, in this section we will deduce some specific properties involving the spaces $L_1^{\alpha,\beta}[-1,1]$ which are pertinent to this work.\ We advise the reader that a general treatment on the Jacobi polynomials can be found in \cite{szego}.

The following elementary property will be used without further mention in Lemma \ref{coef1}, Theorem \ref{mainjacobi} and Theorem \ref{mainjacobi2} ahead.

\begin{lem}\label{inc}
The inclusions $L^{\alpha,\beta}_1[-1,1] \subset L^{\alpha+1,\beta}_1[-1,1]$ and $L^{\alpha,\beta}_1[-1,1] \subset L^{\alpha,\beta+1}_1[-1,1]$ hold for $\alpha, \beta >-1$.

\end{lem}
\begin{proof} We justify the first one, the other one being similar.\ It suffices to observe that
\begin{eqnarray*}
\int_{-1}^1|f(t)|(1-t)^{\alpha+1} (1+t)^\beta dt & = & \int_{-1}^1|f(t)|(1-t)^{\alpha} (1+t)^\beta dt\\
&  & \hspace*{10mm}-\int_{-1}^1|f(t)|t(1-t)^{\alpha} (1+t)^\beta dt,
\end{eqnarray*}
as long as all the integrals exist.\ If $f\in L^{\alpha,\beta}_1[-1,1]$, it is easily seen that both integrals in the right hand side of the equality are finite.\ Hence, $f\in L^{\alpha+1,\beta}_1[-1,1]$.
\end{proof}

The following three-term recurrence formula for Jacobi polynomials is known and plays an important role in the arguments ahead.\ It follows from \cite[p.72]{szego} after the incorporation of the normalization we have adopted for the Jacobi polynomials.

\begin{prop}\label{ch} The normalized Jacobi polynomials $R_n^{\alpha,\beta}$ satisfy
$$
(1-t^2)\frac{d}{dt} R_n^{\alpha,\beta} = A_n^{\alpha,\beta} R_{n-1}^{\alpha,\beta} + B_n^{\alpha,\beta}R_n^{\alpha,\beta} + C_n^{\alpha,\beta} R_{n+1}^{\alpha,\beta},\quad n\geq 1,
$$
where
$$
A_n^{\alpha,\beta}  =  \frac{2n(n+\beta)(n+\alpha+\beta+1)}{(2n+\alpha+\beta)(2n+\alpha+\beta+1)},
$$
$$
B_n^{\alpha,\beta}  = (\alpha-\beta) \frac{2n(n+\alpha+\beta+1)}{(2n+\alpha+\beta)(2n+\alpha+\beta+2)},
$$
and
$$
C_n^{\alpha,\beta}  =  - \frac{2n(n+\alpha+1)(n+\alpha+\beta+1)}{(2n+\alpha+\beta+1)(2n+\alpha+\beta+2)}.
$$
\end{prop}

The next lemma describes a relation among Fourier-Jacobi coefficients of a function $f$ in $L_1^{\alpha,\beta}[-1,1]$
and Fourier-Jacobi coefficients of the same function $f$ as an element of either $L_1^{\alpha+1,\beta}[-1,1]$ or $L_1^{\alpha,\beta+1}[-1,1]$.

\begin{lem}\label{coef1} If $f$ belongs to $L_1^{(\alpha,\beta)}[-1,1]$ and $n\in \mathbb{Z}_+$, then
$$
(\alpha+1)a_n^{\alpha+1,\beta}= \frac{(n+\alpha+1)(n+\alpha+\beta+1)}{2n+\alpha+\beta+1}a_{n}^{\alpha,\beta} - \frac{(n+1)(n+\beta+1)}{2n+\alpha+\beta+3} a_{n+1}^{\alpha,\beta},
$$
and
$$
a_n^{\alpha,\beta+1}= \frac{n+\alpha+\beta+1}{2n+\alpha+\beta+1} a_{n}^{\alpha,\beta} + \frac{n+1}{2n+\alpha+\beta+3}a_{n+1}^{\alpha,\beta}.
$$
\end{lem}
\begin{proof} The proof of the first equality begins with Formula (4.5.4) in \cite{szego}:
$$
(1-x)P_n^{\alpha+1,\beta} = \frac{2(n+\alpha+1)}{2n+\alpha+\beta+2}P_n^{\alpha,\beta} - \frac{2(n+1)}{2n+\alpha+\beta+2}P_{n+1}^{\alpha,\beta}, \quad x\in (-1,1).
$$
Multiplying it by $f(t)(1-t)^\alpha(1+t)^\beta$, integrating and using the orthogonality relation for Jacobi polynomials, we obtain
$$
\frac{h_n^{\alpha+1,\beta}}{P_{n}^{\alpha+1,\beta}(1)}a_n^{\alpha+1,\beta} = \frac{2(n+\alpha+1)}{2n+\alpha+\beta+2}
\frac{h_n^{\alpha,\beta}}{P_{n}^{\alpha,\beta}(1)}a_n^{\alpha,\beta} - \frac{2(n+1)}{2n+\alpha+\beta+2}\frac{h_{n+1}^{\alpha,\beta}}{P_{n+1}^{\alpha,\beta}(1)}a_{n+1}^{\alpha,\beta}.
$$
However, it is easily seen that
$$
\frac{2(n+\alpha+1)}{2n+\alpha+\beta+2}\frac{h_n^{\alpha,\beta}}{h_n^{\alpha+1,\beta}}\frac{P_{n}^{\alpha+1,\beta}(1)}{P_{n}^{\alpha,\beta}(1)} =
\frac{(n+\alpha+1)(n+\alpha+\beta+1)}{(\alpha+1)(2n+\alpha+\beta+1)},\quad n \in \mathbb{Z}_+
$$
and
$$
\frac{2(n+1)}{2n+\alpha+\beta+2}\frac{h_{n+1}^{\alpha,\beta}}{h_n^{\alpha+1,\beta}}\frac{P_{n}^{\alpha+1,\beta}(1)}{P_{n+1}^{\alpha,\beta}(1)} =
\frac{(n+1)(n+\beta+1)}{(\alpha+1)(2n+\alpha+\beta+3)}\quad n \in \mathbb{Z}_+.
$$
Returning these two formulas into what we had before leads to the first equality in the statement of the lemma.\ The proof of the second one is analogous, but needs the recurrence relation
$$
(1+x)P_n^{\alpha,\beta+1}=\frac{2(n+1)}{2n+\alpha+\beta+2}P_{n+1}^{\alpha,\beta}+\frac{2(n+\beta+1)}{2n+\alpha+\beta+2}P_n^{\alpha,\beta},\quad x \in (-1,1).
$$
 The proof is complete.
\end{proof}


\section{Differentiating Jacobi-Fourier expansions}

In this section, we discuss the differentiability of functions that belong to $L_1^{\alpha,\beta}[-1,1]$.\ The first result in this section is purely technical.\ It provides convergence of a numerical sequence of real numbers under a certain control of the sequence.\ It first appeared in \cite{meneg} as a generalization of another one proved in \cite{ziegel}.

\begin{lem}\label{convseq} Let $\{b_n\}$ be a sequence of nonnegative real numbers satisfying
$$
b_n \geq \lambda_n b_{n+1} - \xi_n, \quad n\in \mathbb{Z}_+,
$$
in which $\{\lambda_n\}$ is an increasing sequence of positive real numbers converging to 1 and $\{\xi_n\}$ is a sequence of nonnegative numbers for which $\sum_{n=0}^\infty n\xi_n<\infty$.\ If $\sum_{n=0}^\infty b_n<\infty$ and $\{\lambda_n^n\}$ has a positive lower bound, then $\{nb_n\}$ converges to 0.
\end{lem}

The second result is equally technical.\ It provides the increasingness of a particular sequence that will appear in the proof of the next theorem.

\begin{lem} \label{lambdan} For $\alpha,\beta >-1$, define
$$
\lambda_n:= \frac{n(n+\beta)(2n+\alpha+\beta+2)}{(n+\alpha+1)(n+\alpha+\beta+1)(2n+\alpha+\beta)}, \quad n=1,2,\ldots.
$$
If $2\alpha \geq \alpha+\beta \geq -1$, then $\{\lambda_n\}$ is an increasing sequence converging to 1.\ In addition, the sequence
$\{\lambda_n^n \}$ has a positive lower bound.
\end{lem}
\begin{proof} The function $f:[1,\infty) \to \mathbb{R}$ given by
$$
f(x)=\frac{x(x+\beta)(2x+\alpha+\beta+2)}{(x+\alpha+1)(x+\alpha+\beta+1)(2x+\alpha+\beta)}, \quad x\geq 1,
$$
is positive under the condition $\alpha \geq \beta \geq -1-\alpha$.\ Computing $\ln f(x)$ and using logarithmic differentiation, we reach the following equality:
\begin{eqnarray*}
\frac{f'(x)}{f(x)} & = & \left[\frac{1}{x}-\frac{1}{x+\alpha+\beta+1}+\frac{1}{x+\beta}-\frac{2}{2x+\alpha+\beta}\right.\\
& & \hspace{4cm}\left.+\frac{2}{2x+\alpha+\beta+2}-\frac{1}{x+\alpha+1}\right].
\end{eqnarray*}
A convenient enhancement leads to
\begin{eqnarray*}
\frac{f'(x)}{f(x)} & = & \left[\frac{\alpha+\beta+1}{x(x+\alpha+\beta+1)}+\frac{\alpha-\beta}{(x+\beta)(2x+\alpha+\beta)}\right.\\
& & \hspace{4cm} \left. +\frac{\alpha-\beta}{(2x+\alpha+\beta+2)(x+\alpha+1)}\right].
\end{eqnarray*}
Since all the fractions on the right hand side of the above equation are nonnegative in $[1,\infty)$ when $\alpha \geq \beta \geq -1-\alpha$, it follows that $f'(x)\geq 0$ in $[1,\infty)$.\ The second assertion follows from the fact that $\{\lambda_n^n\}$ is a sequence of positive terms and the obvious equality
$\lim_{n\to \infty}\lambda_n^n =e^{-2\alpha-1}$.
\end{proof}

In the proof of the theorem below, we use the standard decomposition of a real number $r$ through its positive and negative parts: $r=r_+-r_-$.

\begin{thm}\label{mainjacobi} Let $f$ be a function in $L_1^{\alpha,\beta}[-1,1]$.\ If $2\alpha \geq \alpha+\beta \geq -1$, all the Fourier-Jacobi coefficients $a_k^{\alpha,\beta}$ are nonnegative
and both series $\sum_{k=0}^{\infty}a_k^{\alpha,\beta}$ and $\sum_{k=0}^{\infty}|a_k^{\alpha+1,\beta}|$ converge, then $f$ is differentiable in $(-1,1)$ and
$$
(1-t^2)f'(t)=f_1(t)-f_2(t), \quad t \in (-1,1),
$$
in which $f_1$ and $f_2$ are continuous functions of $L_1^{\alpha,\beta}[-1,1]$.\ The Fourier-Jacobi
coefficients of the functions $f_1$ and $f_2$ are nonnegative.
\end{thm}
\begin{proof} Let us assume all the assumptions mentioned in the statement of the theorem.\ The initial step in the proof consists in the differentiation of the function ($N\geq 1)$
$$
f_N(t) := \sum_{n=0}^N a_n^{\alpha,\beta} R_n^{\alpha,\beta}(t), \quad t \in (-1,1),
$$
and the use of Proposition \ref{ch}.\ The result is
\begin{eqnarray*}
(1-t^2)f'_N(t) & = &  \sum_{n=1}^N a_n^{\alpha,\beta}(1-t^2) \frac{d}{dt}R_n^{\alpha,\beta}(t)\\
& = & \sum_{n=1}^N a_n^{\alpha,\beta} \left[ A_n^{\alpha,\beta} R_{n-1}^{\alpha,\beta}(t) +
B_n^{\alpha,\beta}R_n^{\alpha,\beta}(t) + C_n^{\alpha,\beta} R_{n+1}^{\alpha,\beta}(t)\right],
\end{eqnarray*}
that is,
\begin{eqnarray*}
(1-t^2)f'_N(t) & = & 2\sum_{n=1}^N \frac{n(n+\beta)(n+\alpha+\beta+1)}{(2n+\alpha+\beta)(2n+\alpha+\beta+1)} a_n^{\alpha,\beta}  R_{n-1}^{\alpha,\beta}(t)\\
& + &  2\sum_{n=1}^N \frac{(\alpha-\beta)n(n+\alpha+\beta+1)}{(2n+\alpha+\beta)(2n+\alpha+\beta+2)} a_n^{\alpha,\beta} R_n^{\alpha,\beta}(t)\\
& - & 2\sum_{n=1}^N \frac{n(n+\alpha+1)(n+\alpha+\beta+1)}{(2n+\alpha+\beta+1)(2n+\alpha+\beta+2)} a_n^{\alpha,\beta} R_{n+1}^{\alpha,\beta}(t),
\end{eqnarray*}
An adjustment in the indices leads to an expression in the form
$$
(1-t^2)f'_N(t) = 2\left[ F_1^{\alpha,\beta}(t) + G_{N}^{\alpha,\beta}(t) - H_{N}^{\alpha,\beta}(t) + S_{N}^{\alpha,\beta}(t) \right], \quad t \in (-1,1),
$$
where
$$
F_1^{\alpha,\beta} =  \frac{\beta+1}{\alpha +\beta+3}a_1^{\alpha,\beta} R_0^{\alpha,\beta} + \frac{2(\beta+2)(\alpha+\beta+3)}{(\alpha+\beta+4)(\alpha+\beta+5)}a_2^{\alpha,\beta}R_1^{\alpha,\beta},$$
$$
G_{N}^{\alpha,\beta} =  (\alpha-\beta)\sum_{n=1}^N \frac{n(n+\alpha+\beta+1)}{(2n+\alpha+\beta)(2n+\alpha+\beta+2)} a_n^{\alpha,\beta} R_n^{\alpha,\beta},
$$
$$
H_{N}^{\alpha,\beta} =  \sum_{n=N-1}^N \frac{n(n+\alpha+1)(n+\alpha+\beta+1)}{(2n+\alpha+\beta+1)(2n+\alpha+\beta+2)}a_n^{\alpha,\beta}R_{n+1}^{\alpha,\beta},
$$
and
\begin{eqnarray*}
S_{N}^{\alpha,\beta} & = &\sum_{n=2}^{N-1} \left[ \frac{(n+1)(n+\beta+1)(n+\alpha+\beta+2)}{(2n+\alpha+\beta+2)(2n+\alpha+\beta+3)}a_{n+1}^{\alpha,\beta}\right.\\
&  & \hspace{4cm}- \left.\frac{(n-1)(n+\alpha)(n+\alpha+\beta)}{(2n+\alpha+\beta-1)(2n+\alpha+\beta)}a_{n-1}^{\alpha,\beta}\right] R_n^{\alpha,\beta}.
\end{eqnarray*}
An application of the Weierstrass M-test coupled with one of our convergence assumptions implies that $G_{N}^{\alpha,\beta}$ converges to the continuous function
$$
G^{\alpha,\beta}(t)=(\alpha-\beta)\sum_{n=1}^\infty \frac{n(n+\alpha+\beta+1)}{(2n+\alpha+\beta)(2n+\alpha+\beta+2)} a_n^{\alpha,\beta} R_n^{\alpha,\beta}, \quad t \in (-1,1).
$$
as $N\to \infty$, uniformly in $t$.\ Next, we move to
the function $H_{N}^{\alpha,\beta}$, rewriting it in the form
\begin{eqnarray*}
H_{N}^{\alpha,\beta} & = & \frac{N(N+\alpha)(N+\alpha+\beta)}{(2N+\alpha+\beta-1)(2N+\alpha+\beta)}a_{N-1}^{\alpha,\beta}R_{N}^{\alpha,\beta}\\
&  & \hspace{1cm}-\frac{(N+\alpha)(N+\alpha+\beta)}{(2N+\alpha+\beta-1)(2N+\alpha+\beta)}a_{N-1}^{\alpha,\beta}R_{N}^{\alpha,\beta}\\
&  & \hspace{2cm}+\frac{N(N+\alpha+1)(N+\alpha+\beta+1)}{(2N+\alpha+\beta+1)(2N+\alpha+\beta+2)}a_N^{\alpha,\beta}R_{N+1}^{\alpha,\beta}.
\end{eqnarray*}
Since $\lim_{N\to \infty}a_{N-1}^{\alpha,\beta}=0$, the second summand of $H_{N}^{\alpha,\beta}$ approaches 0 as $N \to \infty$, uniformly in $t\in (-1,1)$.\ As for the
first one, we need to recall the equality
\begin{eqnarray*}
\frac{(N+\alpha)(N+\alpha+\beta)}{(2N+\alpha+\beta-1)(2N+\alpha+\beta)}a_{N-1}^{\alpha,\beta}\!\!\! & = & \!\!\!\frac{\alpha+1}{2N+\alpha+\beta}a_{N-1}^{\alpha+1,\beta}\\
\!\! & + & \!\!\! \frac{N(N+\beta)}{(2N+\alpha+\beta)(2N+\alpha+\beta+1)}a_{N}^{\alpha,\beta},
\end{eqnarray*}
which can be extracted from Lemma \ref{coef1}.\ While the coefficients $a_{N-1}^{\alpha,\beta}$ are nonnegative, we emphasize at this point that the same may be not true for the
coefficients $a_{N-1}^{\alpha+1,\beta}$.\ Since we intend to apply Lemma \ref{convseq}, we circumvent this signal inconvenience with the inequality
\begin{eqnarray*}
\frac{(N+\alpha)(N+\alpha+\beta)}{(2N+\alpha+\beta-1)(2N+\alpha+\beta)}a_{N-1}^{\alpha,\beta} \!\!\! & \geq & \!\! -\left[\frac{\alpha+1}{2N+\alpha+\beta}a_{N-1}^{\alpha+1,\beta}\right]_-\\
\!\!\! & + & \!\!\! \frac{N(N+\beta)}{(2N+\alpha+\beta)(2N+\alpha+\beta+1)}a_{N}^{\alpha,\beta}.
\end{eqnarray*}
Now, for $N\geq 1$, we define
$$
b_N := \frac{(N+\alpha)(N+\alpha+\beta)}{(2N+\alpha+\beta-1)(2N+\alpha+\beta)}a_{N-1}^{\alpha,\beta},
$$
$$
\lambda_N:= \frac{N(N+\beta)(2N+\alpha+\beta+2)}{(N+\alpha+1)(N+\alpha+\beta+1)(2N+\alpha+\beta)},
$$
and
$$
\xi_N:=\left[\frac{\alpha+1}{2N+\alpha+\beta}a_{N-1}^{\alpha+1,\beta}\right]_-.
$$
It is easily seen that $\{b_N\}$ is a sequence of nonnegative real numbers satisfying
$b_N \geq \lambda_N b_{N+1} - \xi_N$, $n=1,2,\ldots$.\ Lemma \ref{lambdan} implies that $\{\lambda_N\}$ is a sequence of positive numbers increasing to 1 and that $\{\lambda_N^N\}$ has a positive lower bound.\ Moreover, since $\sum_{n=0}^\infty |a_n^{\alpha+1,\beta}|$ converges, we have that $\sum_{N=1}^\infty N\xi_N < \infty$.\
Obviously, $\sum_{N=1}^{\infty} b_N$ converges due to the convergence of $\sum_{k=0}^{\infty}a_k^{\alpha,\beta}$.\ An application of Lemma \ref{convseq} implies that the first summand of $H_{N}^{\alpha,\beta}$ approaches 0 as $ N \to \infty$, uniformly in $t\in (-1,1)$.\ A similar procedure leads to the same conclusion for the third summand of $H_{N}^{\alpha,\beta}$.\ Thus, $H_{N}^{\alpha,\beta}$ approaches 0 as $ N \to \infty$, uniformly in $t\in (-1,1)$.\ Next, we look at $S_{N}^{\alpha,\beta}$, first using Lemma \ref{coef1} twice to write it like
\begin{eqnarray*}
S_{N}^{\alpha,\beta} & = & \sum_{n=2}^{N-1}\left[\frac{(n+\alpha+1)(n+\alpha+\beta+1)(n+\alpha+\beta+2)}{(2n+\alpha+\beta+1)(2n+\alpha+\beta+2)}a_{n}^{\alpha,\beta}\right.\\
& & \hspace{5cm} - \left. \frac{(\alpha+1)(n+\alpha+\beta+2)}{2n+\alpha+\beta+2}a_n^{\alpha+1,\beta}\right]R_n^{\alpha,\beta}\\
& - & \sum_{n=2}^{N-1} \left[ \frac{n(n-1)(n+\beta)}{(2n+\alpha+\beta)(2n+\alpha+\beta+1)}a_n^{\alpha,\beta}\right.\\
& & \hspace{6.5cm} + \left.\frac{(\alpha+1)(n-1)}{2n+\alpha+\beta}a_{n-1}^{\alpha+1,\beta}\right]R_n^{\alpha,\beta},
\end{eqnarray*}
that is,
\begin{eqnarray*}
S_{N}^{\alpha,\beta} & = &  \sum_{n=2}^{N-1} \left[\frac{(n+\alpha+1)(n+\alpha+\beta+1)(n+\alpha+\beta+2)}{(2n+\alpha+\beta+1)(2n+\alpha+\beta+2)} \right.\\
& &  \hspace{4cm} - \left. \frac{n(n-1)(n+\beta)}{(2n+\alpha+\beta)(2n+\alpha+\beta+1)}\right] a_n^{\alpha,\beta}R_n^{\alpha,\beta}\\
& - & \sum_{n=2}^{N-1}\left[\frac{(\alpha+1)(n+\alpha+\beta+2)}{2n+\alpha+\beta+2}a_n^{\alpha+1,\beta} + \frac{(\alpha+1)(n-1)}{2n+\alpha+\beta}a_{n-1}^{\alpha+1,\beta}\right] R_n^{\alpha,\beta}.
\end{eqnarray*}
The coefficient
\begin{eqnarray*}
Q_n^{\alpha,\beta} & : = & \frac{(n+\alpha+1)(n+\alpha+\beta+1)(n+\alpha+\beta+2)}{(2n+\alpha+\beta+2)(2n+\alpha+\beta+1)}\\
& & \hspace*{35mm} - \frac{n(n-1)(n+\beta)}{(2n+\alpha+\beta)(2n+\alpha+\beta+1)}
\end{eqnarray*}
can be written as a quotient of two polynomials of degree 2 in the variable $n$.\ Hence, due to our convergence assumptions, an application of the Weierstrass $M$-test is all that is needed in order to see that
the first summand of $S_{N}^{\alpha,\beta}$ converges as $N \to \infty$, uniformly in $t\in (-1,1)$.\ A much easier argument now using the convergence of $\sum_{k=0}^{\infty}|a_k^{\alpha+1,\beta}|$ reveals that
the second summand converges likewise.\ We conclude that $S_{N}^{\alpha,\beta} $ converges to
\begin{eqnarray*}
\sum_{n=2}^{\infty} Q_n^{\alpha,\beta} a_n^{\alpha,\beta}R_n^{\alpha,\beta} & - & \sum_{n=2}^{\infty}\left[\frac{(\alpha+1)(n+\alpha+\beta+2)}{2n+\alpha+\beta+2}a_n^{\alpha+1,\beta}\right.\\
& & \hspace{2.5cm} \left. + \frac{(\alpha+1)(n-1)}{2n+\alpha+\beta}a_{n-1}^{\alpha+1,\beta}\right] R_n^{\alpha,\beta},
\end{eqnarray*}
as $N \to \infty$, uniformly in $t\in (-1,1)$.\ Clearly, the functions $f_1$ and $f_2$ in the first assertion of the theorem are given by
$$
f_1=F_1^{\alpha,\beta}+G^{\alpha,\beta}+\sum_{n=2}^{\infty} Q_n^{\alpha,\beta} a_n^{\alpha,\beta}R_n^{\alpha,\beta}
$$
and
$$
f_2=\sum_{n=2}^{\infty}\left[\frac{(\alpha+1)(n+\alpha+\beta+2)}{2n+\alpha+\beta+2}a_n^{\alpha+1,\beta} + \frac{(\alpha+1)(n-1)}{2n+\alpha+\beta}a_{n-1}^{\alpha+1,\beta}\right] R_n^{\alpha,\beta}.
$$
The Fourier-Jacobi coefficients of $f_2$ are obviously nonnegative.\ On the other hand, the Fourier-Jacobi coefficients of
$G^{\alpha,\beta}$ are all nonnegative due to the assumption $\alpha \geq \beta$.\ In order to conclude the proof of the theorem, it suffices to verify that the Fourier-Jacobi coefficients of $f_1-F_1^{\alpha,\beta}-G^{\alpha,\beta}$ are nonnegative.\ That amounts to showing
that $Q_n^{\alpha,\beta}\geq 0$, $n\geq 2$, that is,
$$
\frac{(n+\alpha+1)(n+\alpha+\beta+1)(n+\alpha+\beta+2)}{n(n-1)(n+\beta)}\geq \frac{2n+\alpha+\beta+2}{2n+\alpha+\beta}, \quad n\geq 2.
$$
But that is obvious from the inequality
$$
\frac{n+(\alpha+\beta)/2}{n+\beta}\ \frac{n+\alpha+\beta}{n-1}\ \frac{n+\alpha+\beta+2}{n}\ \frac{n+\alpha+1}{n+1+(\alpha+\beta)/2}\geq 1,
$$
which is true when $n\geq 2$ and our assumptions on $\alpha$ and $\beta$ are in force.
\end{proof}


\section{Differentiating Fourier-Jacobi expansions II}

The main goal in this section is the deduction of a version of Theorem \ref{mainjacobi} where both indices $\alpha$ and $\beta$ vary.\ It is needed
due to the path we have chosen to prove Theorem \ref{main}-$(iv)$.\ Due to the similarity with the results in the previous section, some of the details will be omitted.

We begin with a four-term relation which can be deduced by a simple combination of those in Lemma \ref{coef1}.

\begin{lem}\label{coef3}If $f$ belongs to $L_1^{\alpha,\beta}[-1,1]$, then
\begin{eqnarray*}
(\alpha+1) a_n^{\alpha+1,\beta+1} & = & \frac{(n+\alpha+1)(n+\alpha+\beta+2)(n+\alpha+\beta+1)}{(2n+\alpha+\beta+2)(2n+\alpha+\beta+1)}a_n^{\alpha,\beta}\\
& + & (\alpha-\beta)\frac{(n+1)(n+\alpha+\beta+2)}{(2n+\alpha+\beta+2)(2n+\alpha+\beta+4)}a_{n+1}^{\alpha,\beta}\\
& - & \frac{(n+1)(n+2)(n+\beta+2)}{(2n+\alpha+\beta+4)(2n+\alpha+\beta+5)}a_{n+2}^{\alpha,\beta}.
\end{eqnarray*}
\end{lem}

The next lemma is a different version of Lemma \ref{lambdan}.

\begin{lem} \label{lambdan2} For $\alpha,\beta >0$ and $n=1,2,\ldots$, define
$$
\lambda_n:= \frac{n(n+1)(n+\beta+1)(2n+\alpha+\beta+4)}{(n+\alpha+2)(n+\alpha+\beta+1)(n+\alpha+\beta+2)(2n+\alpha+\beta+2)}.
$$
If $2\alpha \geq \alpha+\beta \geq -1$, then $\{\lambda_n\}$ is an increasing sequence converging to 1.\ In addition, $\{\lambda_n^n\}$ has a positive lower bound.
\end{lem}
\begin{proof} It is similar to the proof of Lemma \ref{lambdan}, but requires a re-definition of the function $f$.\ Following the steps of that proof and adjusting $f$, we have that
\begin{eqnarray*}
\frac{f'(x)}{f(x)} & = & \left[\frac{\alpha+\beta+1}{x(x+\alpha+\beta+1)}+\frac{\alpha-\beta}{(x+\beta+1)(2x+\alpha+\beta+2)}\right.\\
& & \hspace*{1cm} \left. +\frac{\alpha-\beta}{(2x+\alpha+\beta+4)(x+\alpha+2)}+ \frac{\alpha+\beta+1}{(x+1)(x+\alpha+\beta+2)}\right],
\end{eqnarray*}
while $\lim_{n\to \infty}\lambda_n^n =e^{-2\alpha-\beta -2}$.
\end{proof}

The counterpart of Theorem \ref{mainjacobi} we need here is as follows.

\begin{thm}\label{mainjacobi2} Let $f$ be a function in $L_1^{\alpha,\beta}[-1,1]$.\ If $2\alpha \geq \alpha+\beta \geq -1$, the coefficients $a_k^{\alpha,\beta}$ are nonnegative and both series $\sum_{k=0}^{\infty}a_k^{\alpha,\beta}$ and $\sum_{k=0}^{\infty}|a_k^{\alpha+1,\beta+1}|$ converge, then $f$ is differentiable in $(-1,1)$ and
$$
(1-t^2)f'(t)=f_1(t)-f_2(t), \quad t \in (-1,1),
$$
in which $f_1$ and $f_2$ are continuous functions of $L_1^{\alpha,\beta}[-1,1]$.\ The Fourier-Jacobi
coefficients of the functions $f_1$ and $f_2$ are nonnegative.
\end{thm}
\begin{proof} Suppose all the assumptions mentioned in the statement of the theorem hold.\ We follow the steps in the proof of Theorem \ref{mainjacobi}, decomposing the same function $f_N$
through the very same functions $F_1^{\alpha,\beta}$, $G_N^{\alpha,\beta}$, $H_N^{\alpha,\beta}$.\ The convergence property obtained for $G_N^{\alpha,\beta}$ and for the second summand of
 $H_N^{\alpha,\beta}$ persists here due to the common assumption $\sum_{k=0}^{\infty}a_k^{\alpha,\beta}<\infty$.\ Significant changes begin in the analysis of convergence for the first summand of
 $H_N^{\alpha,\beta}$.\ Indeed, in that case we set
$$
b_N := \frac{(N+\alpha)(N+\alpha+\beta)}{(2N+\alpha+\beta-1)(2N+\alpha+\beta)}a_{N-1}^{\alpha,\beta},
$$
and invoke Lemma \ref{coef3} to write
\begin{eqnarray*}
b_N & = &  \frac{\alpha+1}{N+\alpha+\beta+1}a_{N-1}^{\alpha+1,\beta+1}\\
& - & (\alpha-\beta)\frac{N(N+\alpha+\beta+1)}{(N+\alpha+\beta+1)(2N+\alpha+\beta+2)(2N+\alpha+\beta)}a_N^{\alpha,\beta}\\
& + & \frac{N(N+1)(N+\beta+1)}{(N+\alpha+\beta+1)(2N+\alpha+\beta+2)(2N+\alpha+\beta+3)}a_{N+1}^{\alpha,\beta}.
\end{eqnarray*}
In the next step we prepare $b_N$ in accordance with Lemma \ref{convseq}.\ It is promptly seen that
$$
b_N  \geq  -[\xi_N]_- +\frac{N(N+1)(N+\beta+1)}{(N+\alpha+\beta+1)(2N+\alpha+\beta+2)(2N+\alpha+\beta+3)}a_{N+1}^{\alpha,\beta},
$$
where
\begin{eqnarray*}
\xi_N & := & \frac{\alpha+1}{N+\alpha+\beta+1}a_{N-1}^{\alpha+1,\beta+1}\\
& - & (\alpha-\beta)\frac{N(N+\alpha+\beta+1)}{(N+\alpha+\beta+1)(2N+\alpha+\beta+2)(2N+\alpha+\beta)}a_N^{\alpha,\beta}.
\end{eqnarray*}
Setting
$$
\lambda_N:= \frac{N(N+1)(N+\beta+1)(2N+\alpha+\beta+4)}{(N+\alpha+2)(N+\alpha+\beta+1)(N+\alpha+\beta+2)(2N+\alpha+\beta+2)},
$$
it is now seen that $\{b_N\}$ is a sequence of nonnegative real numbers satisfying
$$
b_N \geq \lambda_N b_{N+2} - \xi_N, \quad N=1,2,\ldots.
$$
Lemma \ref{lambdan2} implies that $\{\lambda_N\}$ is a sequence of positive numbers increasing to 1 and $\{\lambda_N^N\}$ has a positive lower bound.\ Moreover, if
both series $\sum_{k=0}^{\infty}a_k^{\alpha,\beta}$ and $\sum_{n=0}^\infty |a_n^{\alpha+1,\beta+1}|$ converge, we have that $\sum_{N=1}^\infty N\xi_N < \infty$.\
Obviously, $\sum_{N=1}^{\infty} b_N$ converges whenever $\sum_{k=0}^{\infty}a_k^{\alpha,\beta}$ does.\ Hence, under the conditions in the statement of the theorem, an
 application of Lemma \ref{convseq} implies that the first summand of $H_{N}^{\alpha,\beta}$ approaches 0 as $ N \to \infty$, uniformly in $t\in (-1,1)$.\ A similar procedure leads to the same conclusion for the third summand of $H_{N}^{\alpha,\beta}$.\ Thus, $H_{N}^{\alpha,\beta}$ approaches 0 as $ N \to \infty$, uniformly in $t\in (-1,1)$.\ Next, we handle the convergence of $S_{N}^{\alpha,\beta}$.\ Using Lemma \ref{coef3}, it is not hard to see that
\begin{eqnarray*}
S_{N}^{\alpha,\beta} & = & \sum_{n=2}^{N-1}\left[\frac{(n+1)(n+\beta+1)(n+\alpha+\beta+2)}{(2n+\alpha+\beta+2)(2n+\alpha+\beta+3)}\right.\\
&   & \hspace*{0.6cm} -\left.\frac{n(n-1)(n+1)(n+\beta+1)}{(n+\alpha+\beta+1)(2n+\alpha+\beta+2)(2n+\alpha+\beta+3)}\right]a_{n+1}^{\alpha,\beta}R_n^{\alpha,\beta}\\
& + & (\alpha-\beta) \sum_{n=2}^{N-1} \frac{n(n-1)(n+\alpha+\beta+1)}{(n+\alpha+\beta+1)(2n+\alpha+\beta+2)(2n+\alpha+\beta)}a_n^{\alpha,\beta}R_n^{\alpha,\beta}\\
&  & \hspace*{40mm}-(\alpha+1) \sum_{n=2}^{N-1} \frac{n-1}{n+\alpha+\beta+1} a_{n-1}^{\alpha+1,\beta+1}R_n^{\alpha,\beta}.
\end{eqnarray*}
Now, observe that the summand
\begin{eqnarray*}
Q_n^{\alpha,\beta} & := & \frac{(n+1)(n+\beta+1)(n+\alpha+\beta+2)}{(2n+\alpha+\beta+2)(2n+\alpha+\beta+3)}\\
&  & \hspace*{20mm} -\frac{n(n-1)(n+1)(n+\beta+1)}{(n+\alpha+\beta+1)(2n+\alpha+\beta+2)(2n+\alpha+\beta+3)}
\end{eqnarray*}
is a quotient of polynomials of degree 3 in the variable $n$.\ Therefore, taking into account the convergence of the numerical series in the statement of the theorem, an application of the Weierstrass M-test shows that $S_{N}^{\alpha,\beta} $ converges to
$$
\sum_{n=2}^{\infty} Q_n^{\alpha,\beta} a_{n+1}^{\alpha,\beta}R_n^{\alpha,\beta} + K^{\alpha,\beta}  - (\alpha+1) \sum_{n=2}^{\infty} \frac{n-1}{n+\alpha+\beta+1} a_{n-1}^{\alpha+1,\beta+1}R_n^{\alpha,\beta}.
$$
as $N \to \infty$, uniformly in $t\in (-1,1)$, in which
$$
K^{\alpha,\beta}:=(\alpha-\beta) \sum_{n=2}^{\infty} \frac{n(n-1)(n+\alpha+\beta+1)}{(n+\alpha+\beta+1)(2n+\alpha+\beta+2)(2n+\alpha+\beta)}a_n^{\alpha,\beta}R_n^{\alpha,\beta}.
$$
Clearly, the functions $f_1$ and $f_2$ in the first assertion of the theorem are given by
$$
f_1=F_1^{\alpha,\beta}+G^{\alpha,\beta}+K^{\alpha,\beta}+\sum_{n=2}^{\infty} Q_n^{\alpha,\beta} a_n^{\alpha,\beta}R_n^{\alpha,\beta}
$$
and
$$
f_2=(\alpha+1) \sum_{n=2}^{\infty} \frac{n-1}{n+\alpha+\beta+1} a_{n-1}^{\alpha+1,\beta+1}R_n^{\alpha,\beta}.
$$
This completes the proof of the first assertion of the theorem.\ The Fourier-Jacobi coefficients of $f_2$ are obviously nonnegative.\ On the other hand, the Fourier-Jacobi coefficients of
$G^{\alpha,\beta}$ are all nonnegative due to the assumption $\alpha \geq \beta$.\ The proof will be complete as long as we show that
the Fourier-Jacobi coefficients of $f_1-F_1^{\alpha,\beta}-G^{\alpha,\beta}-K^{\alpha,\beta}$ are nonnegative.\ That boils down to proving that
$$
\frac{(n+\alpha+\beta+1)(n+\alpha+\beta+2)}{(n-1)n} \geq 1,\quad n=2,3,\ldots.
$$
But this is certainly true under the assumption $2\alpha \geq \alpha+\beta \geq -1$.
\end{proof}


\section{Differentiability of positive definite kernels on $\mathbb{M}^d$}

In this section, we present a proof for Theorem \ref{main} and discuss a few issues regarding the theorem.\ Here, the symbol $A \hookrightarrow B$ will indicate the existence
of an isometric embedding from the metric space $A$ into the metric space $B$.\ The lemma below includes all the isometric embeddings among the spaces pertaining to this paper
that will be needed in the proofs ahead (see \cite[p.66]{Askey} and references therein).

\begin{lem}\label{embed} There exist isometric embeddings as below:\\
$(i)$ $S^d \hookrightarrow S^{d+1}$, $d=1,2,\ldots$;\\
$(ii)$ $\mathbb{P}^{d}(\mathbb{R}) \hookrightarrow \mathbb{P}^{d+1}(\mathbb{R})$, $d=2,3,\ldots$;\\
$(iii)$ $\mathbb{P}^{d}(\mathbb{C}) \hookrightarrow \mathbb{P}^{d+2}(\mathbb{C})$, $d=4,6,\ldots$;\\
$(iv)$ $\mathbb{P}^{d}(\mathbb{H}) \hookrightarrow \mathbb{P}^{d+4}(\mathbb{H})$, $d=8,12,\ldots$;\\
$(v)$ $\mathbb{P}^{d}(\mathbb{R}) \hookrightarrow \mathbb{P}^{2d}(\mathbb{C})$, $d=2,3,\ldots$;\\
$(vi)$ $\mathbb{P}^{2d}(\mathbb{C}) \hookrightarrow \mathbb{P}^{4d}(\mathbb{H})$, $d=2,3,\ldots$;\\
$(vii)$ $\mathbb{P}^{8}(\mathbb{H}) \hookrightarrow \mathbb{P}^{16}(Cay)$.
\end{lem}

We begin with the case of positive definite kernels on real projective spaces.

\begin{thm}  \label{main1} Let $d$ be an integer at least 3.\ If $K$ is a continuous, positive definite and zonal kernel on $\mathbb{P}^d(\mathbb{R})$, then the radial part $K_r^d$ of $K$ is continuously differentiable on $(-1,1)$.\ The derivative $(K_r^d)'$ satisfies
$$
(1-t^2)(K_r^d)'(t) = f_1(t) - f_2(t),\quad t\in(-1,1),
$$
in which $f_1$ and $f_2$ are the radial parts of two continuous, positive definite and zonal kernels on $\mathbb{P}^{d-2}(\mathbb{R})$.
\end{thm}
\begin{proof} Let $K$ be a continuous, positive definite and zonal kernel on $\mathbb{P}^d(\mathbb{R})$.\ If $d\geq 4$, Lemma \ref{embed}-$(ii)$ provides the isometric embedding $\mathbb{P}^{d-2}(\mathbb{R}) \hookrightarrow \mathbb{P}^d(\mathbb{R})$.\ Hence, the kernel $(x,y) \in K_r^d(\cos |xy|/2)$ is continuous, positive definite and zonal on $\mathbb{P}^{d-2}(\mathbb{R})$ as well.\ Gangolli's characterization for positive definiteness described in the introduction now yields that
$$K_r^d \in L^{(d-2)/2,-1/2}_1[-1,1]\cap L_1^{(d-4)/2,-1/2}[-1,1].$$
In particular, the assumptions in Theorem \ref{mainjacobi} are satisfied with $\alpha=(d-4)/2$ and $\beta=-1/2$ and the assertion of the theorem follows.\ Next, we consider the case $d=3$.\ We have the embedding $\mathbb{P}^{1}(\mathbb{R}) \hookrightarrow \mathbb{P}^3(\mathbb{R})$ but $\mathbb{P}^{1}(\mathbb{R})$ does not belong to our initial list of spaces.\ We circumvent this inconvenience observing that $\mathbb{P}^1(\mathbb{R})$ is isometric isomorphic to the circle $S^1_{1/2}$ of $\mathbb{R}^2$, centered at 0 of radius $1/2$.\ A characterization for the continuous, positive definite and zonal kernels on $S^1_{1/2}$ can be obtained from that one on $S^1$, introducing a dilation of 2 in Gangolli's characterization (the coefficients in the expansion do not change).\ The arguments used in the case $d\geq 4$ can then be repeated taking into account small changes produced by the dilation.
\end{proof}

The next theorem complements the previous one.

\begin{thm}  \label{main2} Let $d$ be an integer at least 4.\ If $K$ is a continuous, positive definite and zonal kernel on $\mathbb{P}^{d}(\mathbb{C})$, then the radial part $K_r^d$ of $K$ is continuously differentiable on $(-1,1)$.\ The derivative $(K_r^d)'$ satisfies
$$
(1-t^2)(K_r^d)'(t) = f_1(t) - f_2(t),\quad t\in(-1,1),
$$
in which $f_1$ and $f_2$ are the radial parts of two continuous, positive definite and zonal kernels on $\mathbb{P}^{d-2}(\mathbb{C})$.
\end{thm}
\begin{proof} It is a repetition of the procedure used in the proof of the previous theorem, now using Lemma \ref{embed}-$(iii)$.\ The case $d=4$ needs to be treated separately taking into account the fact that $\mathbb{P}^2(\mathbb{C})$ is isometric isomorphic to the sphere $S^2_{1/2}$ in $\mathbb{R}^3$, centered at 0 of radius $1/2$ (\cite[p. 88]{gasqui}).
\end{proof}

Next, we move to the case of the quaternionic projective spaces.

\begin{thm}  \label{main3} Let $d$ be an integer at least 8.\ If $K$ is a continuous, positive definite and zonal kernel on $\mathbb{P}^{d}(\mathbb{H})$, then the radial part $K_r^d$ of $K$ is continuously differentiable on $(-1,1)$.\ The derivative $(K_r^d)'$
satisfies a relation of the form
$$
(1-t^2)(K_r^d)'(t) = f_1(t) - f_2(t),\quad t\in(-1,1),
$$
in which $f_1$ and $f_2$ are the radial parts of two continuous, positive definite and zonal kernels on $\mathbb{P}^{d/2-2}(\mathbb{C})$, if $d \in 8\mathbb{Z}_++8$ and on $\mathbb{P}^{d/2}(\mathbb{C})$, if $d\in 8\mathbb{Z}_++12$.
\end{thm}
\begin{proof} Let $K$ be a continuous, positive definite and zonal kernel on $\mathbb{P}^d(\mathbb{H})$.\ Since $d \in 4\mathbb{Z}_+$, we will consider two cases.\ If $d \in 8\mathbb{Z}_++12$, then $d/2 +2 \in 4\mathbb{Z}_+$.\ In particular, since $d/2 +2 <d$, then $\mathbb{P}^{d/2+2}(\mathbb{H}) \hookrightarrow \mathbb{P}^d(\mathbb{H})$ by Lemma \ref{embed}-$(iv)$.\ Likewise, due to Lemma \ref{embed}-$(vi)$, $\mathbb{P}^{d/2}(\mathbb{C}) \hookrightarrow \mathbb{P}^d(\mathbb{H})$.\ Returning to Gangolli's characterization once again, we can conclude that $K_r^d \in L_1^{d/4,1}[-1,1] \cap L_1^{d/4-1,0}[-1,1]$.\ In this case, the assumptions of Theorem \ref{mainjacobi2} are satisfied with $\alpha=d/2+2$ and $\beta=1$.\ An application of that result yields the assertion of the theorem in the first case.\ If $d \in 8\mathbb{Z}_+ +16$, then the isometric embeddings are  $\mathbb{P}^{d/2}(\mathbb{H})\hookrightarrow  \mathbb{P}^d(\mathbb{H})$ and $\mathbb{P}^{d/2-2}(\mathbb{C}) \hookrightarrow\mathbb{P}^d(\mathbb{H})$.\ Thus, $K_r^d \in L_1^{d/4-1,1}[-1,1] \cap L_1^{d/4-2,0}[-1,1]$ and the very same procedure leads to the assertion of the theorem once again.\ Finally, if $d=8$, we need to employ a procedure similar to that used at the end of the proof of the previous theorem, but using the fact that $\mathbb{P}^4(\mathbb{H})$ is isometric isomorphic to the sphere $S^4_{1/2}$ of $\mathbb{R}^5$, centered at 0 of radius $1/2$ (\cite[p. 88]{gasqui}).
\end{proof}

Finally, here is what our methodology provides in the case of the Cayley projective plane.

\begin{thm}  \label{main4} If $K$ is a continuous, positive definite and zonal kernel on $\mathbb{P}^{16}(Cay)$, then the radial part $K_r^{16}$ of $K$ is continuously differentiable on $(-1,1)$.\ The derivative $(K_r^{16})'$
satisfies a relation of the form
$$
(1-t^2)(K_r^{16})'(t) = f_1(t) - f_2(t),\quad t\in(-1,1),
$$
in which $f_1$ and $f_2$ are the radial parts of two continuous, positive definite and zonal kernels on $S^2$.
\end{thm}
\begin{proof} Let $K$ be a continuous, positive definite and zonal kernel on $\mathbb{P}^{16}(Cay)$.\ Combining the last three assertions in Lemma \ref{embed}, it is promptly seen that $\mathbb{P}^4(\mathbb{H})\hookrightarrow P^{16}(Cay)$.\ On other hand, $\mathbb{P}^4(\mathbb{H})$ is isometric isomorphic to the sphere $S^4_{1/2}$ while $S^2_{1/2} \hookrightarrow S^4_{1/2}$.\ Thus, proceeding as before, we have that $K_r^{16} \in L_1^{1,1}[-1,1] \cap L_1^{0,0}[-1,1]$.\ An application of Theorem \ref{mainjacobi2} with $\alpha=\beta=0$ leads to the assertion of the theorem.
\end{proof}


\section{Final remarks}

If we start with a continuous, positive definite and zonal kernel $K$ on $S^d$, it is not hard to see, via Lemma \ref{embed}-$(i)$, that the radial part $K_r$ of $K$ belongs to both, $L_1^{(d-2)/2, (d-2)/2}[-1,1]$ and $L_1^{(d-4)/2,(d-4)/2}[-1,1]$.\ An application of Theorem \ref{mainjacobi2} with $\alpha=\beta=(d-2)/2$ leads to Theorem 4.1 in \cite{ziegel}.\ Thus, Theorem \ref{main} is in fact an extension of the later to compact two
point homogeneous spaces.

After we apply one of the theorems from the previous section to a certain kernel, the resulting functions $f_1$ and $f_2$ in the decomposition of the derivative of the radial part of the kernel end up being the radial parts of positive definite kernels on a compact two point homogeneous space of dimension lower than the dimension of the original one.\ In particular, we may apply one of the theorems to the functions $f_1$ and $f_2$ in order to reach higher order derivatives for the radial part of the original kernel and so on.\
The process ends with the exhaustion of the dimension of the original compact two point homogeneous space.\ A careful analysis of this procedure leads to the following extension of Theorem \ref{main}.

\begin{thm}
The following properties regarding the differentiability on $(-1,1)$, of the radial part $K_r^d$ of a continuous, positive definite and zonal kernel $K$ on $\mathbb{M}^d$, hold:\\
$(i)$ Sphere $S^d$: $K_r^d$ is of class $C^{\lfloor(d-1)/2\rfloor}$;\\
$(ii)$ Real projective spaces $\mathbb{P}^d(\mathbb{R})$: $K_r^d$ is of class $C^{\lfloor(d-1)/2\rfloor}$;\\
$(iii)$ Complex projective spaces $\mathbb{P}^d(\mathbb{C})$: $K_r^d$ is of class $C^{(d-2)/2}$;\\
$(iv)$ Quaternionic projective spaces $\mathbb{P}^d(\mathbb{H})$: $K_r^d$ is of class $C^{(d-4)/4}$ if $d \in 8\mathbb{Z}_++8$, and of class $C^{d/4}$ if $d\in 8\mathbb{Z}_++12$;\\
$(v)$ Cayley projective plane $\mathbb{P}^{16}(Cay)$: $K_r^{16}$ is of class $C^{1}$.
\end{thm}

Despite our efforts, it remains open at this time whether or not the orders of differentiability mentioned in the previous theorem are the best possible ones.\ In the case $(i)$, this was ratified for $d$ odd (see the last section of \cite{ziegel}).\ However, the same methodology seems not to apply to the other cases in the theorem.


\section*{Acknowledgement} The first author was partially supported by CAPES.\ The second one by FAPESP, under grant 2014/00277-5.





\begin{thebibliography}{99}


\bibitem{Askey} R. Askey, Orthogonal polynomials and special functions.\ Society for Industrial and Applied Mathematics, Philadelphia, Pa., 1975.

\bibitem{gangolli} R. Gangolli, Positive definite kernels on homogeneous spaces and certain stochastic processes related to Lévy's Brownian motion of several parameters.\ Ann. Inst. H. Poincaré Sect. B (N.S.) 3 (1967), 121-226.

\bibitem{gasqui}  J. Gasqui, H. Goldschmidt, Radon transforms and the rigidity of the Grassmannians.\ Annals of Mathematics Studies, 156. Princeton University Press, Princeton, NJ, 2004.

\bibitem{gneiting} T. Gneiting, Strictly and non-strictly positive definite functions on spheres.\ Bernoulli 19 (2013), no. 4, 1327-1349.


\bibitem{kush1} A. Kushpel, S.A. Tozoni, Entropy and widths of multiplier operators on two-point homogeneous spaces.\ Constr. Approx. 35 (2012), no. 2, 137-180.


\bibitem{meneg} V.A. Menegatto, Differentiability of bizonal positive definite kernels on complex spheres.\ J. Math. Anal. Appl. 412 (2014), no. 1, 189-199.

\bibitem{platonov} S.S. Platonov, On some problems in the theory of the approximation of functions on compact homogeneous manifolds.\ (Russian) {\em Mat. Sb.} 200 (2009), no. 6, 67--108; translation in Sb. Math. 200 (2009), no. 5-6, 845--885.


\bibitem{szego} G. Szeg\"{o}, Orthogonal polynomials.\ Fourth edition. American Mathematical Society, Colloquium Publications, Vol. XXIII, American Mathematical Society, Providence, R.I., 1975.

\bibitem{wang} Wang Hsien-Chung, Two-point homogeneous spaces.\ Ann. Math. 55 (1952), no. 2, 177-191.

\bibitem{ziegel} J. Ziegel, Convolution roots and differentiability of isotropic positive definite functions on spheres.\ Proc. Amer. Math. Soc. 142 (2014), no. 6, 2063-2077.
\end{thebibliography}



%
%

\vspace*{2cm}

\noindent V. S. Barbosa and V. A. Menegatto \\
Departamento de
Matem\'atica,\\ ICMC-USP - S\~ao Carlos, Caixa Postal 668,\\
13560-970 S\~ao Carlos SP, Brasil\\ e-mails: victorrsb@gmail.com; menegatt@icmc.usp.br

\end{document}